\documentclass[a4paper,12pt,reqno]{amsart}
\usepackage[body={17cm,21cm}]{geometry}
\usepackage[initials]{amsrefs}

\usepackage{amssymb,amscd}
\usepackage[all]{xy}

\usepackage[mathscr]{eucal}
\usepackage{enumerate}

\numberwithin{equation}{section}
\theoremstyle{plain}
\newtheorem{thm}[equation]{Theorem}
\newtheorem{cor}[equation]{Corollary}
\newtheorem{lem}[equation]{Lemma}
\newtheorem{prop}[equation]{Proposition}

\newtheorem{definition}[equation]{Definition}

\theoremstyle{definition}

\theoremstyle{remark}

\DeclareMathOperator{\Br}{Br}

\DeclareMathOperator{\codim}{codim}

\DeclareMathOperator{\Div}{Div}

\DeclareMathOperator{\Gal}{Gal}

\DeclareMathOperator{\im}{Im}

\DeclareMathOperator{\Pic}{Pic}

\DeclareMathOperator{\rank}{rank}
\DeclareMathOperator{\Res}{Res}

\def\Br{{\rm Br}}
\def\inv{{\rm inv}}

\def\hom{{\rm Hom}}

\def\Gal{{\rm Gal}}

\def\ker {{\rm  Ker}}
\def\hom{{\rm Hom}}

\def\Pic{{\rm Pic}}
\def\ker{{\rm ker}}

\DeclareFontEncoding{OT2}{}{} 
   \newcommand{\textcyr}[1]{%
     {\fontencoding{OT2}\fontfamily{wncyr}\fontseries{m}\fontshape{n}%
      \selectfont #1}}

\newcommand{\Sha}{{\mbox{\textcyr{Sh}}}}

\DeclareFontFamily{U}{wncy}{}
\DeclareFontShape{U}{wncy}{m}{n}{%
<5>wncyr5%
<6>wncyr6%
<7>wncyr7%
<8>wncyr8%
<9>wncyr9%
<10>wncyr10%
<11>wncyr10%
<12>wncyr6%
<14>wncyr7%
<17>wncyr8%
<20>wncyr10%
<25>wncyr10}{}
\DeclareMathAlphabet{\cyr}{U}{wncy}{m}{n}
\begin{document}

\title[]
{Strong approximation with Brauer-Manin obstruction for groupic varieties}

\author{Yang CAO}

\address{Yang CAO \newline Laboratoire de Math\'ematiques d'Orsay
\newline Univ. Paris-Sud, CNRS, Univ. Paris-Saclay \newline 91405 Orsay, France}

\email{yang.cao@math.u-psud.fr}

\author{Fei XU}

\address{Fei XU \newline School of Mathematical Sciences, \newline Capital Normal University,
\newline 105 Xisanhuanbeilu, \newline 100048 Beijing, China}

\email{xufei@math.ac.cn}


\date{\today.}



\maketitle

\begin{abstract} Strong approximation with Brauer-Manin obstruction is established for smooth varieties containing  a connected linear algebraic group as an open subset with a compatible action.
\end{abstract}

\tableofcontents

\section{Introduction}

Classical strong approximation for semi-simple simply connected linear algebraic groups has been established by Eichler in \cite{Eichler1}-\cite{Eichler2}, Weil in \cite{Weil}, Shimura in \cite{Shimura}, Kneser in \cite{Kneser}, Platonov in \cite{Platonov1}-\cite{Platonov2}, Prasad in \cite{Prasad} and so on from thirties to seventies of last century. Min\v{c}hev in \cite{Mincev} pointed out that classical strong approximation is not true for varieties which are not simply connected. Colliot-Th\'el\`ene and the second named author in \cite{CTX} first suggested that one should study strong approximation with Brauer-Manin obstruction which generalizes classical strong approximation by using Manin's idea and established strong approximation with Brauer-Manin obstruction for homogenous spaces of semi-simple linear algebraic groups with application to integral points. Since then, Harari in \cite{Ha08} proved strong approximation with Brauer-Manin obstruction for tori and Demarche in \cite{D} extended Harari's result to connected linear algebraic groups and Wei and the second named author in \cite{WX} extended Harari's result to groups of multiplicative type.  Borovoi and Demarche in \cite{BD} established strong approximation with Brauer-Manin obstruction for homogenous spaces of connected linear algebraic groups with connected stabilizers. Colliot-Th\'el\`ene and the second named author in \cite{CTX1} proved strong approximation with Brauer-Manin obstruction for certain families of quadratic forms and  Colliot-Th\'el\`ene and Harari in \cite{CTH} extended this result to certain families of homogenous spaces of linear algebraic groups.

 In our previous paper \cite{CX}, strong approximation with Brauer-Manin obstruction has been established for open toric varieties. It is natural to ask whether such a result is still true if the torus is replaced by a connected linear algebraic group.
The basic idea in \cite{CX} is to construct the standard toric varieties (see Definition 2.12 in \cite{CX}) by using the complement divisors.
 The group action provides the crucial relation of local integral points for almost all places (see Proposition 4.1 in \cite{CX}).
 In order to prove strong approximation, one further needs to show that any point outside the torus can be approximated by a point in the torus with the same local invariant for all elements in the Brauer groups of a given toric variety (see Proposition 4.2 in \cite{CX}). The proof of such local approximation property is reduced to an affine toric variety case.

 Such a method cannot be generalized to the case of arbitrary connected linear algebraic groups directly. For example, one cannot expect that such varieties can be covered by affine pieces of the same type varieties for a general linear algebraic group even over an algebraically closed field (see \cite{Su}).

Instead of explicit constructions, we apply the descent theory to study universal torsors. Combining with the rigidity property of torsors under multiplicative groups developed by Colliot-Th\'el\`ene in \cite{CT08}, we conclude that such torsors also contain linear algebraic groups with compatible action (see Lemma \ref{a_Y}). Applying the idea of group action in \cite{CX}, one can prove strong approximation with Brauer-Manin obstruction using the relation of local integral points at almost all places (see Lemma \ref{loc-int}) provided by these torsors. In fact, this paper recovers the main result in \cite{CX} by this new method.

Notation and terminology are standard. Let $k$ be a number field, $\Omega_k$ the set of all primes in $k$ and  $\infty_k$ the set of all Archimedean primes in $k$. Write $v<\infty_k$ for $v\in \Omega_k\setminus \infty_k$. Let $O_k$ be the ring of integers of $k$ and $O_{k,S}$ the $S$-integers of $k$ for a finite set $S$ of $\Omega_k$ containing $\infty_k$. For each $v\in \Omega_k$, the completion of $k$ at $v$ is denoted by $k_v$ and the completion of $O_k$ at $v$ by $O_v$. Write $O_{v}=k_v$ for $v\in \infty_k$ and $k_{\infty}=\prod_{v\in \infty_k} k_v$. Let ${\mathbf A}_k$ be the adelic ring of $k$ and ${\mathbf A}_k^f$ the finite adelic ring of $k$.

For any scheme $X$ over $k$, we denote $X_{\bar k}=X\times_k \bar{k}$ with $\bar{k}$ a fixed algebraic closure of $k$. A variety $X$ over $k$ is defined to be a reduced separated scheme of finite type over $k$. Let $$\Br(X)=H_{\text{\'et}}^2(X, \Bbb G_m),  \ \ \ \Br_1(X)= \ker[\Br(X)\rightarrow \Br(X_{\bar{k}})], \ \ \ \Br_a(X)=\Br_1(X)/\Br(k) . $$  We denote by $\Bbb A_k^n$ the affine space of dimension $n$ over $k$.
Define
$$ X(\mathbf A_k)_{\bullet}= [\prod_{v\in \infty_k} \pi_0(X(k_v))]\times X(\mathbf A_k^f) $$
where $\pi_0(X(k_v))$ is the set of connected components of $X(k_v)$ for each $v\in \infty_k$. Since an element in $\Br (X)$ takes a constant value at each connected component of $\pi_0(X(k_v))$ for all $v\in \infty_k$, one can define
$$ X({\mathbf A}_k)_{\bullet}^B = \{ (x_v)_{v\in \Omega_k}\in X({\mathbf A}_k)_{\bullet}: \ \ \sum_{v\in \Omega_k} \inv_v(\xi(x_v))=0, \ \ \forall \xi\in B \}  $$ for any subset $B$ of $\Br(X)$. Class field theory implies that $X(k) \subseteq X(\mathbf A_k)_{\bullet}^B$.

\begin{definition} \label{sa} Let $k$ be a number field. Let $X$ be a variety over $k$.

(1) If $X(k)$ is dense in $ X({\mathbf A}_k)_{\bullet}$, we say $X$ satisfies strong approximation off $\infty_k$.

(2) If $X(k)$ is dense in $ X({\mathbf A}_k)_{\bullet}^{B}$ for some subset $B$ of $\Br(X)$, we say $X$ satisfies strong approximation with respect to $B$ off $\infty_k$.
\end{definition}

In this paper, we will study strong approximation for a $G$-groupic variety for a connected linear algebraic group $G$ defined as follows.

\begin{definition} \label{tv} Let $k$ be a field. Let $G$ be a connected linear algebraic group over $k$ and $X$ be a variety over $k$.

(1)  $X$ is called a $G$-variety if there is an action of $G$
$$ a_X: G\times_k X \longrightarrow X  $$ over $k$.

A morphism from a $G$-variety $X$ to a $G$-variety $X'$ is defined to be a morphism of schemes from $X$ to $X'$ which is compatible with the actions of $G$. Such a morphism is called $G$-morphism.

(2) If $X$ is a geometrically integral $G$-variety and $G$ is contained in $X$ as an open subset such that the action $a_X|_{G\times_k G}=m_G$ where $m_G$ is the multiplication of $G$, then we call $X$ a $G$-groupic variety.

A morphism $f$ from a $G$-groupic variety $X$ to a $G'$-groupic variety $X'$ is defined to be a morphism of schemes from $X$ to $X'$ such that $f|_{G}: G\rightarrow G'$ is a homomorphism of linear algebraic groups.
\end{definition}

It is clear that a morphism of groupic varieties is compatible with group actions.

 Let $k$ be a field with characteristic $0$. For any connected linear algebraic group $G$, the reductive part $G^{red}$ of $G$ is given by $$ 1\rightarrow R_u(G)\rightarrow G \rightarrow G^{red} \rightarrow 1 $$ where $R_u(G)$ is the unipotent radical of $G$. Let $G^{ss}=[G^{red}, G^{red}]$ be the semi-simple part of $G$, let $G^{sc}$ be the semi-simple simply connected covering of $G^{ss}$, let $G^{tor}$ be the maximal quotient torus of $G$ and let $\varsigma_G: G\rightarrow  G^{tor} $ be the canonical quotient homomorphism.

 The main result of this paper is the following theorem.

\begin{thm} \label{main} Let $k$ be a number field and $G$ a connected linear algebraic group. Assume $G'(k_\infty)$ is not compact for any non-trivial simple factor $G'$ of $G^{sc}$. Then any smooth $G$-groupic variety $X$ over $k$ satisfies strong approximation with respect to $\Br_1(X)$ off $\infty_k$. \end{thm}

The paper is organized as follows. In \S 2, we study some basic properties of $G$-varieties over a field of characteristic 0. In \S 3, we apply the descent theory developed by Colliot-Th\'el\`ene and Sansuc in \cite{CTS87} and the rigidity property of torsors under multiplicative groups developed by Colliot-Th\'el\`ene in \cite{CT08} to construct the right candidates so that one can expect the arguments analogue to those of \cite{CX} to apply. All results in this section work over arbitrary fields of characteristic 0 as well. In \S 4, we give a proof of Theorem \ref{main} based on the results in previous sections. In \S 5, we study strong approximation with Brauer-Manin obstruction off any finite non-empty subset of $\Omega_k$ and prove such strong approximation when the invertible functions over $\bar k$ are constant (see Theorem \ref{arch-main}).

\section{Preliminary on $G$-varieties}

In this section, we establish some basic results on $G$-varieties which we need in the next sections.  In this section, we assume that $k$ is an arbitrary field $k$ with $char(k)=0$ and $\Gamma_k=\Gal(\bar{k}/k)$ where $\bar k$ is an algebraic closure of $k$.

First, we need a kind of Stein factorization in the category of $G$-varieties.

\begin{lem}\label{gr-nor} Let $A\xrightarrow {\lambda} B$ be a dominant $G$-morphism of normal and geometrically integral $G$-varieties over $k$ where $G$ is a connected linear algebraic group. Assume $B$ is affine. Then $\lambda$ can be factorized into morphisms of $G$-varieties $A\xrightarrow{\iota} C$ and $C\xrightarrow{\tau} B$ such that $C$ is normal and geometrically integral, the generic fiber of $\iota$ is geometrically integral and $\tau$ is finite.
\end{lem}

\begin{proof} Since $\lambda$ is dominant, one can view $k(B)$ as a subfield of $k(A)$. Let $k(C)$ be the algebraic closure of $k(B)$ inside $k(A)$ and $k[C]$ be the integral closure of $k[B]$ inside $k(C)$. Since $A$ is normal, one has $k[C]\subseteq k[A]$ and $k[C]$ is integral closure of $k[B]$ inside $k[A]$.  Since $A$ is geometrically integral, one has $k$ is algebraically closed inside $k(A)$. Therefore $k$ is algebraically closed in $k(C)$. Then $C=Spec(k[C])$ is normal, geometrically integral and $C \rightarrow B$ is finite. Moreover $\lambda$ factors into $A\xrightarrow{\iota} C$ and $C\xrightarrow{\tau} B$ by inclusion of the global sections.

Similarly, we can factor $A\times_k G \xrightarrow{\lambda\times id} B\times_k G$ into $G\times_k A \rightarrow \widetilde{C} \rightarrow G\times_k B$ such that $k[\widetilde{C}]$ is the integral closure of $k[B\times_k G]$ inside $k[A\times_k G]$. Then one has the following commutative diagram
 \[ \begin{CD} G\times_k A  @>{pr_2}>> A \\
 @VVV @VVV \\
\widetilde{C}  @>>> C  \\
@VVV @VVV \\
G\times_k B @>>{pr_2}> B
 \end{CD} \]
where $\widetilde{C}  \rightarrow  C$ is a canonical morphism induced by $k[A]\hookrightarrow k[G\times_k A]$. Moreover one has a unique morphism $ \widetilde{C} \xrightarrow{\theta} G\times_k C$ which is finite because both the morphism $\widetilde{C} \rightarrow G\times_k B$ and the morphism $G\times_k C \rightarrow G\times_k B$ are finite. Let $\eta_C$ be the generic point of $C$. Then one obtains
$$ G\times_k A_{\eta_C} \rightarrow {\widetilde{C}}_{\eta_C} \rightarrow G\times_k k(C) $$
over $\eta_C$. Since $k(C)$ is algebraic closed in $k(A)$, one obtains that all fibers of $$ G\times_k A_{\eta_C} \rightarrow G\times_k k(C)$$ are geometrically integral. Therefore $ {\widetilde{C}}_{\eta_C} \rightarrow G\times_k k(C) $ is an isomorphism. This implies that $\theta$ is an isomorphism.

Replacing $pr_2$ by the actions $a_A$ and $a_B$ in the above diagram, one obtains the following commutative diagram
\[ \begin{CD} G\times_k A  @>{a_A}>> A \\
 @VVV @VVV \\
\widetilde{C}=G\times_k C  @>{a_C}>> C  \\
@VVV @VVV \\
G\times_k B @>>{a_B}> B
 \end{CD} \]
where $a_C$ is induced by the homomorphism of the global sections which is the unique homomorphism to make the above diagram commute. This implies that $C$ is a $G$-variety and $\iota$ and $\tau$ are morphisms of $G$-varieties by uniqueness. \end{proof}

We can apply this lemma to prove the following result.

\begin{prop}\label{fib-int} Let $A\xrightarrow {\lambda} B$ be a  $G$-morphism of geometrically integral $G$-varieties over $k$. Assuming $B$ is affine and smooth. If $B=G/G_1$ where both $G$ and $G_1$ are connected linear algebraic groups, then all fibers of $\lambda$ are nonempty and geometrically integral. Moreover if $A$ is smooth, then $\lambda$ is smooth.
\end{prop}

\begin{proof} Without loss of generality, one can assume $k=\bar k$. Since the action of $G$ on $B$ is transitive, $\lambda$ is surjective.

Suppose $A$ is smooth. By applying Lemma \ref{gr-nor}, one can factorize $\lambda$ into morphisms of $G$-varieties $A\xrightarrow{\iota} C$ and $C\xrightarrow{\tau} B$ such that $C$ is geometrically integral, the generic fiber of $\iota$ is geometrically integral and $\tau$ is finite. Since $B$ has only a single orbit, any orbit of $G$ in $C$ contains the generic point of $C$. This implies that $C$ has a single orbit of $G$ as well. Since $G_1$ is connected, $G_1$ contains no proper closed subgroups of finite index. This implies $B=C$. This means that the generic fiber of $\lambda$ is nonempty and geometrically integral.

In general, let $A^{sm}$ be the smooth locus of $A$. Then $A^{sm}$ is also a $G$-variety. By the above result, one has that the generic fiber of $\lambda_{A^{sm}}$ is geometrically integral. Since $A_{\eta_B}^{sm}$ is open dense in $A_{\eta_B}$ where $\eta_B$ is the generic point of $B$, one concludes that $A_{\eta_B}$ is geometrically integral. Since all fibers are translated by the group action, one concludes that all fibers of $\lambda$ are nonempty and geometrically integral. By generic smoothness (see Corollary 10.7 of Chapter III in \cite{Hartshorne}), one further obtains that $\lambda$ is smooth.
\end{proof}

\begin{prop}\label{Uu}
Let $A$ be a smooth geometrically integral $G$-variety, $B\subset A$ an open $G$-subvariety. Then there exists an open $G$-subvariety $C\subset A$, such that $\codim(A\setminus C, A)\geq 2$, $B\subset C$, $(C\setminus B)_{\bar k}\cong \coprod_iD_i$ and each $D_i$ is a smooth integral $G_{\bar{k}}$-variety with $\dim(D_i)=\dim(A)-1$.
\end{prop}

\begin{proof}
Let $C'= A\setminus [(A\setminus B)_{sing}]$, where $(A\setminus B)_{sing}$ is the singular part of $A\setminus B$.
Then $C'$ is an open $G$-subvariety of $A$, $\codim(A\setminus C',A)\geq 2$ and $C'\setminus B$ is smooth over $k$.
Thus $(C'\setminus B)=C_1\coprod C_2$ where $C_1$ is the union of all codimension $1$ connected components of $C'\setminus B$, and $C_2$ is the union of all codimension $\geq 2$ connected components of $C'\setminus B$.
 Then $C_1$ and $C_2$ are stable under the action of $G$.
Let $C:=C'\setminus C_2$ and one obtains the result.
\end{proof}

\section{Pull-back of universal torsors over smooth compactifications}

Harari and Skorobogatov have extended the descent theory of Colliot-Th\'el\`ene and Sansuc in \cite{CTS87} to open varieties by defining the extended type of torsors in \cite{HaSk}. One could try to use this generalisation, but we will use the pull-back of the universal torsors of smooth compactifications of open varieties. By the rigidity property of torsors under multiplicative groups developed in \cite{CT08}, one concludes that such torsors also contain a linear algebraic group as an open subset with a compatible action.

In this section, we assume that $k$ is an arbitrary field $k$ with $char(k)=0$ and $\Gamma_k=\Gal(\bar{k}/k)$ where $\bar k$ is an algebraic closure of $k$.
Let $X$ be a smooth $G$-groupic variety over $k$ and $X^c$ be a smooth compactification of $X$ over $k$.  Then $X^c$ is a smooth compactification of $G$ over $k$ and $\Pic(X^c_{\bar{k}})$ is a flasque $\Gamma_k$-module by Theorem 3.2 in \cite{BKG}.
Let $T$ be a torus over $k$ such that the character group $$T^*=Hom_{\bar{k}}(T, \Bbb G_m)=\Pic(X^c_{\bar{k}}) . $$

By corollary 2.3.9 in \cite{Sko}, there is a universal torsor $\rho: Z \rightarrow X^c$ under $T$ over $k$ satisfying $\rho^{-1}(1_G)(k) \neq \emptyset$.
Since $\bar {k} [X^c]^{\times}=\bar{k}^{\times}$, by \cite{CTS87} Section 2.1, one has $Pic(Z_{\bar k})=0$  and $\bar {k}[Z]^\times = \bar {k}^\times$.
 Let $$H=Z\times_{X^c} G \subset Z. $$
 Then $H$ is a quasi-trivial linear algebraic group over $k$ (i.e. $Pic(H_{\bar{k}})=0$ and $\bar{k}[H]^{\times}/\bar{k}^{\times}$ is a permutation $Gal(\bar{k}/k)$-module, see Definition 2.1 in \cite{CT08}) and the projection map $\rho_G$ induces a flasque resolution
\begin{equation} \label{flasque} 1 \rightarrow T \xrightarrow{\kappa} H \xrightarrow{\rho_G} G \rightarrow 1 \end{equation} of $G$ by Theorem 5.4 in \cite{CT08}. Moreover, one has
\begin{equation} \label{uh}  \bar k[H]^\times/\bar k^\times \cong  \Div_{Z_{\bar k} \setminus H_{\bar k}} (Z_{\bar k}) \cong \Div_{X^c_{\bar k}\setminus G_{\bar k}}(X^c_{\bar k}) \end{equation}  by
Theorem 1.6.1 in \cite{CTS87} and Lemma B.1 in \cite{CT08}.

The pull-back of the universal torsor $Z \rightarrow X^c$ defines a torsor over $X$ under $T$:
 $$ \rho_X: \ Y=Z \times_{X^c} X \longrightarrow X .$$
 By Proposition 5.1 in \cite{CT08}, the variety $Y$ is quasi-trivial (see Definition 1.1 in \cite{CT08}) and
 \begin{equation} \label{uy} \bar k[Y]^\times/\bar k^\times \cong  \Div_{Z_{\bar k} \setminus Y_{\bar k}} (Z_{\bar k}) \cong \Div_{X^c_{\bar k}\setminus X_{\bar k}}(X^c_{\bar k}) \end{equation}  by
Theorem 1.6.1 in \cite{CTS87} and Lemma B.1 in \cite{CT08}.

\begin{lem}\label{a_Y} The multiplication $m_H: H\times_k H \rightarrow H$ can be extended to an action
$$ a_Y: \ H\times_k Y \rightarrow Y $$ over $k$.
\end{lem}

\begin{proof} One only needs to modify the argument in Theorem 5.6 in \cite{CT08}  and replace $$m_G: G\times_k G \rightarrow G \ \ \ \text{ by } \ \ \ a_X: G\times_k X \rightarrow X$$ with $a_X|_{G\times_k G}=m_G$. By Lemma 5.5 in \cite{CT08}, there is a morphism
$$ a_Y: \ H\times_k Y \rightarrow Y  $$ such that the following diagram

\[ \begin{CD}
H \times_k Y @ >{a_Y}>> Y \\
@ V{\rho_G\times \rho_X}VV  @ VV{\rho_X}V  \\
G \times_k X @>>{a_X}> X
\end{CD}
\]
commutes and $ a_Y|_{H\times_k H}= m_H $. Since $H$ is dense in $Y$, the associativity of $m_H$ implies that $a_Y$ is an action of $H$.  \end{proof}

It is clear that the following diagram
 $$ \begin{CD} @. @. 1 @. 1 \\
@. @. @ VVV @VVV\\
1 @>>> R_u(G) @>>>   \ker(\varsigma_G)  @ >>> G^{ss} @>>> 1\\
@. @V{\cong}VV   @VVV  @VVV \\
1 @>>> R_u(G) @>>> G @>>> G^{red} @>>> 1 \\
@. @. @V{\varsigma_G}VV @VV{\varsigma_{G^{red}}}V \\
@. @. G^{tor} @>{\cong}>> (G^{red})^{tor} \\
@. @. @VVV @VVV \\
@.@. 1 @. 1
\end{CD} $$
commutes and that its columns and rows are exact. Therefore $ker(\varsigma_G)$ is geometrically integral whenever $G$ is connected.

Since
$$ (H^{tor})^*= \bar k[H]^\times /\bar k^\times= \Div_{X_{\bar k} \setminus G_{\bar k}}(X_{\bar k}) \oplus \Div_{X^c_{\bar k}\setminus X_{\bar k}}(X^c_{\bar k})$$  as $\Gamma_k$-module by (\ref{uh}), one has $H^{tor}=T_0\times_k T_1$ where $T_0$ and $T_1$ are tori over $k$ such that $$T_0^*=\Div_{X_{\bar k} \setminus G_{\bar k}}(X_{\bar k}) \cong \Div_{Y_{\bar k} \setminus H_{\bar k}}(Y_{\bar k}) \ \ \ \text{and} \ \ \ T_1^*= \Div_{X^c_{\bar k}\setminus X_{\bar k}}(X^c_{\bar k})$$
by Lemma B.1 in \cite{CT08}.
Moreover, the inclusion $\iota_0: T_0\hookrightarrow H^{tor}$ is induced by
\begin{equation}\label{iota0} \iota_0^*: (H^{tor})^*= \bar k[H]^\times /\bar k^\times\xrightarrow{div} \Div_{Y_{\bar k} \setminus H_{\bar k}}(Y_{\bar k})\cong T_0^*.\end{equation}

\begin{lem} \label{h_0} Let $\psi$ be the surjective homomorphism $H\xrightarrow{\psi} T_1$ obtained by composing $\varsigma_H$ with the projection on $T_1$ and $H_0=ker(\psi)$. Then $H_0$  and $\ker(\varsigma_{H_0})$ are connected and quasi-trivial with the canonical isomorphisms
$$H_0^{ss}\xrightarrow{\cong} H^{ss}\xrightarrow{\cong}G^{sc} \ \ \ \text{and} \ \ \  \chi^*:  T_0^*\xrightarrow{\cong} \bar k[H_0]^\times /\bar k^\times . $$  \end{lem}

\begin{proof} By P.94 in \cite{CT08}, one has $H^{ss}\cong H^{sc}\xrightarrow{\cong} G^{sc}$. It is clear that there is a surjective homomorphism $H_0\xrightarrow{\chi} T_0$ over $k$ such that the following diagram

$$  \begin{CD} @. 1 @. 1 \\
@. @VVV  @VVV\\
@. \ker(\chi) @>{\cong}>> \ker(\varsigma_H)   \\
@.    @VVV  @VVV \\
1 @>>> H_0 @>>> H @>{\psi}>> T_1 @>>> 1 \\
@. @V{\chi}VV @VV{\varsigma_{H}}V @VV{=}V \\
1@>>> T_0 @>{\iota_0}>> H^{tor} @>{\psi}>> T_1 @>>> 1 \\
@. @VVV @VVV \\
@. 1 @. 1
\end{CD} $$
commutes and has exact rows and columns. One concludes that $\ker(\chi)$ is connected and $\ker(\chi)\cong \ker(\varsigma_{H_0})$.
Therefore $H_0$ is connected and $H_0^{ss}\xleftarrow{\cong}\ker(\chi)^{ss}\xrightarrow{\cong}H^{ss}$.
Since $\Pic(H_{\bar{k}})=0$, one further has that $\Pic((H_0)_{\bar k})=0$, and then $\Pic(\ker(\varsigma_{H_0})_{\bar k}))=0$ by (6.11.4) in \cite{Sansuc}. \end{proof}

\begin{lem}\label{constant}
If $\bar{k}[X]^\times =\bar k^\times$, then $H_0\xrightarrow{\rho_G} G$ is surjective and its kernel is a group of multiplicative type.
\end{lem}
\begin{proof} Since $\bar{k}[X]^\times /\bar k^\times=1$, the morphism 
$\bar{k}[G]^\times/\bar{k}^\times \xrightarrow{div}  
\Div_{X_{\bar k} \setminus G_{\bar k}}(X_{\bar k})$ is injective. 
Since one has the following commutative diagram
\[ \begin{CD} \bar{k}[G]^\times /\bar k^\times @>{\rho_G^*}>> \bar {k}[H]^\times/\bar k^\times @>>> \bar{k}[H_0]^\times/\bar k^\times \\
@V{div}VV @VVV @AA{\chi^*}A \\
\Div_{X_{\bar k} \setminus G_{\bar k}}(X_{\bar k}) @>>{\rho_X^*}>  \Div_{Y_{\bar k} \setminus H_{\bar k}}(Y_{\bar k}) @>>{\cong}> T_0^* \end{CD} \] by Lemma \ref{h_0}, one obtains that 
$\bar{k}[G]^\times /\bar k^\times \xrightarrow{\rho_G^*} \bar{k}[H_0]^\times/\bar k^\times$ is injective. By tracing the following diagram with the exact rows
\[ \begin{CD}
1 @>>> \bar{k}[G]^\times /\bar k^\times @>{\rho_G^*}>> \bar {k}[H]^\times/\bar k^\times @>\kappa^*>> T^* \\
@. @. @VV{=}V @. \\
1 @>>> T_1^* @>>{\psi^*}> \bar {k}[H]^\times/\bar k^\times @>>> \bar{k}[H_0]^\times/\bar k^\times @>>> 1 ,\end{CD} \] one obtains that $\kappa^*\circ \psi^*$ is injective. Therefore $T\xrightarrow{\psi\circ \kappa} T_1$ is surjective. By tracing the following diagram with the exact rows
\[ \begin{CD}
1 @>>> T @>{\kappa}>> H @>\rho_G>> G @>>> 1 \\
@. @. @VV{=}V @. \\
1 @>>> H_0 @>>> H @>>{\psi}> T_1 @>>> 1 ,\end{CD} \] one obtains that $H_0\xrightarrow{\rho_G} G $ is surjective and its kernel is a group of multiplicative type.   
\end{proof}

\begin{prop} \label{br} One has the following exact sequence
$$ 1 \rightarrow  \Br_1(X)\rightarrow \Br_1(G) \rightarrow \Br_a (H_0) $$ for a smooth $G$-groupic variety $(G\hookrightarrow X)$ over $k$, where the map $\Br_1(G)\rightarrow \Br_a(H_0)$ is given by the map $H_0\rightarrow H\xrightarrow{\rho_G} G$.
\end{prop}

\begin{proof} On the one hand, one has the following commutative diagram of exact sequences
\[ \begin{CD}
1 @ >>>  \Br_1(X) @ >>> \Br_1(G)  @ >>>  H^2(k, \Div_{X_{\bar k} \setminus G_{\bar k}}(X_{\bar k})) \\
@.  @ V{\rho_X^*}VV @ V{\rho_G^*}VV @ VV{\cong}V \\
1 @ >>> \Br_1(Y) @>>> \Br_1(H) @ >>> H^2(k, \Div_{Y_{\bar k}\setminus H_{\bar k}} (Y_{\bar k}))
\end{CD}
\]
by (6.1.3) of Lemma 6.1 in \cite{Sansuc} and functoriality.

On the other hand, since $\Pic((H_0)_{\bar k})=\Pic(H_{\bar k})=0$ by Lemma \ref{h_0}, the homomorphism $H_0 \rightarrow H$ induces the following commutative diagram
\[ \begin{CD}H^2(k, {H^{tor}}^*) @>{\cong}>> H^2(k, \bar k[H]^\times/\bar{k}^{\times}) @>{\cong}>>  \Br_a(H) \\
 @V{\iota_0^*}VV@VVV @VVV \\
  H^2(k,T_0^*)@>{\chi^*}>{\cong}> H^2(k, \bar k[H_0]^\times/\bar{k}^{\times})@>>{\cong}> \Br_a(H_0)
 \end{CD} \]  by Hochschild-Serre spectral sequence and  Lemma \ref{h_0}.
Since the composition of morphisms
$$H^2(k, \bar k[H]^\times/\bar{k}^{\times})\rightarrow  \Br_a(H)\rightarrow H^2(k, \Div_{Y_{\bar k}\setminus H_{\bar k}} (Y_{\bar k}))$$
 is induced by $\bar k[H]^\times/\bar{k}^{\times}\xrightarrow{div} \Div_{Y_{\bar k}\setminus H_{\bar k}} (Y_{\bar k})$, the result follows from (\ref{iota0}).
 \end{proof}

\begin{lem} \label{ext} The homomorphism $\psi$ in Lemma \ref{h_0} can be extended to a smooth $H$-morphism
$Y \xrightarrow{\psi} T_1$ over $k$ with geometrically integral fibers. \end{lem}

\begin{proof} Let
$$ B_Y=\{ b\in \bar k[Y]^\times : \ b(1_H)=1 \} \ \ \ \text{and} \ \ \  B_H=\{ b \in \bar k[H]^\times : b(1_H)=1 \} $$ be two $\Gamma_k$-modules. Then
$ \bar k[Y]^\times = \bar k^\times \oplus B_Y$ and $ \bar k[H]^\times = \bar k^\times \oplus B_H $ with $B_Y\subset B_H$ by $H\subset Y$ and
 $$ T_1^*  = \Div_{X^c_{\bar k}\setminus X_{\bar k}}(X^c_{\bar k}) \cong \bar k[Y]^\times/\bar k^\times \cong B_Y $$ as $\Gamma_k$-module by (\ref{uy}).

Since $\psi$ induces the injective homomorphism
$$ \psi^*: \ \bar k[T_1]^\times/\bar k^\times \rightarrow \bar k[H]^\times /\bar k^\times  $$ of $\Gamma_k$-module which is compatible with $B_Y\subset B_H$, one concludes that the homomorphism of $\bar k$-algebras $\psi^*: \bar k[T_1] \rightarrow \bar k[H]$ factors through
$$ \bar k[T_1]= \bar k[B_Y] \rightarrow \bar k [Y] \subset \bar k[H] $$
which is also $\Gamma_k$-equivariant.

Since $\psi$ is a homomorphism from $H$ to $T_1$, one has the following commutative diagram
\[ \begin{CD} H\times_k Y  @>{a_Y}>> Y \\
 @V{\psi\times \psi}VV @VV{\psi}V \\
 T_1\times_k T_1 @>>{m_{T_1}}> T_1
 \end{CD} \]
by Lemma \ref{a_Y}. This implies that $\psi$ is an $H$-morphism. By Proposition \ref{fib-int}, one concludes that $\psi$ is smooth with geometrically integral fibers.
\end{proof}

\begin{prop}\label{y_0} Let $$Y_0=\psi^{-1}(1_{T_1}) \subset Y$$ be the fiber of $1_{T_1}$ in $Y\xrightarrow{\psi} T_1$.  Then $Y_0$ is a smooth $H_0$-groupic variety, the map given by divisors of functions
 $$ \bar{k}[H_0]^\times/\bar k^\times  \xrightarrow{div} \Div_{(Y_0)_{\bar k} \setminus (H_0)_{\bar k} } ((Y_0)_{\bar k})$$ is an isomorphism,
and
$$ \bar k[Y_0]^\times =\bar k^\times, \ \ \ \Pic((Y_0)_{\bar k})=0 \ \ \ \text{and} \ \ \ \Br_a(Y_0)=0 . $$
\end{prop}

\begin{proof} By Lemma \ref{ext}, $Y_0$ is an $H_0$-groupic variety. Since $\psi$ is smooth, one concludes that $Y_0$ is smooth.

By the cartesian diagram
 \[ \begin{CD} H_0  @>>> Y_0 \\
 @VVV @VVV \\
 H @>>> Y
 \end{CD} \]
where the horizontal maps are open immersions and the vertical maps are closed immersions, one obtains the commutative diagram of exact sequences
\[ \begin{CD}
1 @>>> \bar{k}[Y]^\times @>>> \bar{k}[H]^\times  @>>> \Div_{Y_{\bar k}\setminus H_{\bar k}}(Y_{\bar k}) @>>> \Pic(Y_{\bar k}) @>>> \Pic(H_{\bar k})  \\
@.  @VVV @VVV @VV{\phi}V @VVV @VVV \\
1 @>>> \bar{k}[Y_0]^\times @>>> \bar{k}[H_0]^\times   @>>> \Div_{(Y_0)_{\bar k} \setminus (H_0)_{\bar k} } ((Y_0)_{\bar k}) @>>> \Pic((Y_0)_{\bar k}) @>>> \Pic((H_0)_{\bar k})
 \end{CD} \]
by Theorem 1.6.1 in \cite{CTS87}, where $\phi$ is the pull-back of Cartier divisors (see Section 2.3 in \cite{Ful}), which are the same as Weil divisors by smoothness.

Let $D$ be an irreducible component of $Y_{\bar k}\setminus H_{\bar k}$. Since $D(\bar k)$ is stable under the action of $H(\bar k)$, one has $$H(\bar k)\cdot (D(\bar k)\cap Y_0(\bar k))\subseteq D(\bar k). $$ For any $x\in D(\bar k)$, there is $h\in H(\bar k)$ such that $\psi(h)=\psi(x)$. Therefore  $$h^{-1}x\in Y_0(\bar k)\cap D(\bar k) \ \ \ \text{and} \ \ \ H(\bar k)\cdot (D(\bar k)\cap Y_0(\bar k))= D(\bar k). $$ This implies that $\phi$ is injective.

On the other hand, for an irreducible component $D_0$ of $(Y_0)_{\bar k} \setminus (H_0)_{\bar k}$, one has $$D_0(\bar k )\cap H(\bar k)=\emptyset $$ and the Zariski closure $\overline{H\cdot D_0}$ of $H \cdot D_0$ is an irreducible closed subset in $Y_{\bar k}\setminus H_{\bar k}$.  Let $D=\overline{H\cdot D_0}$. Then $ D_0 \subseteq  D \cap Y_0$
and $D$ is $H$-invariant. Applying Proposition \ref{fib-int} for morphism $D \xrightarrow{\psi_D} T_1$ of $H$-varieties, one obtains that $\psi_D^{-1}(1_{T_1})=D\cap Y_0$ is geometrically integral. By the maximality of $D_0$, one has $D_0 = D \cap Y_0$. By the action of $H$, one has
$$ \codim(D,Y)=\codim(D_{\eta_{T_1}},Y_{\eta_{T_1}})=\codim(D\cap Y_0, Y_0)=\codim(D_0,Y_0)=1  $$ and $D$ is a divisor of $Y$.  This implies $\phi$ is an isomorphism.

Since $\Pic(Y_{\bar k})=0$, one obtains that the map $$ \bar{k}[H_0]^\times/\bar k^\times  \rightarrow \Div_{(Y_0)_{\bar k} \setminus (H_0)_{\bar k} } ((Y_0)_{\bar k})$$ induced in the above commutative diagram is surjective. This implies that $$\Pic((Y_0)_{\bar k})=0$$ by Lemma \ref{h_0}. Since both $\bar{k}[H_0]^\times/\bar k^\times$ and $\bar{k}[Y_0]^\times/\bar k^\times$ are free abelian groups of finite rank and
$$ \rank(\bar{k}[H_0]^\times/\bar k^\times)=\rank(T_0^*)=\rank( \Div_{Y_{\bar k}\setminus H_{\bar k}}(Y_{\bar k})) = \rank(\Div_{(Y_0)_{\bar k} \setminus (H_0)_{\bar k} } ((Y_0)_{\bar k}))$$ by Lemma \ref{h_0}, (\ref{uh}) and (\ref{uy}), one concludes that the above induced map is an isomorphism. Therefore $ \bar k[Y_0]^\times =\bar k^\times$. Applying the Hochschild-Serre spectral sequence (see Lemma 2.1 in \cite{CTX}), one has $\Br_a(Y_0)=0$. \end{proof}

\begin{definition} \label{st-tv} Let $K$ be a finite \'etale algebra over $k$. The unique minimal toric subvariety $V$ of $(Res_{K/k}(\Bbb G_m)\hookrightarrow Res_{K/k}({\Bbb A}^1))$ over $k$ with respect to $Res_{K/k}(\Bbb G_m)$ such that $$\codim( Res_{K/k}({\Bbb A}^1)\setminus V, Res_{K/k}({\Bbb A}^1))\geq 2$$ is called the standard toric variety of $K/k$.
\end{definition}

Such a standard toric variety always exists by Proposition 2.10 in \cite{CX}.  By Lemma \ref{h_0}, one has $H_0^{tor} \cong T_0$ over $k$.

\begin{prop}\label{chi} Let $T_0$, $H_0$ and $Y_0$ be as above. There exists a standard toric variety $(T_0\hookrightarrow V)$ and an open $H_0$-subvariety $U\subset Y_0$ such that $H_0\subset U$,
$\codim(Y_0\setminus U,Y_0)\geq 2$ and the canonical quotient map $H_0\xrightarrow{\varsigma_{H_0}}T_0$ can be extended a morphism $U\xrightarrow{\varsigma_U} V$ of groupic varieties. Moreover, the morphism $\varsigma_{U}$  is smooth with nonempty and geometrically integral fibres.
\end{prop}

\begin{proof}
By Proposition \ref{Uu},  there exists an open $H_0$-subvariety $U\subset Y_0$, such that $H_0\subset U$, $\codim(Y_0\setminus U, Y_0)\geq 2$, $(U\setminus H_0)_{\bar k}\cong \coprod_iD_i$ and each $D_i$ is a smooth integral $(H_0)_{\bar{k}}$-variety with $\dim(D_i)=\dim(Y_0)-1$.
One notes
$$U_i=(U)_{\bar k}\setminus (\bigcup_{j\neq i} D_j).$$

Since the extension of the morphism if it exists is unique, one can assume that $k=\bar k$. Since
$$ T_0^*\cong \bar k[H_0]/\bar k^\times  \cong \Div_{Y_0 \setminus H_0} (Y_0)\cong \Div_{U \setminus H_0} (U)$$
by Lemma \ref{h_0}, Proposition \ref{y_0} and $\codim(Y_0\setminus U, Y_0)\geq 2$, one has
$$ x_i \in \bar k[H_0]/\bar k^\times  = T_0^* \ \ \ \text{ such that } \ \ \ div_{U}(x_i)=D_i $$ for $1\leq i\leq s$. Then $\{x_1, \cdots, x_s \}$ is a $\Bbb Z$-basis of $T_0^*$ and $$V=\bigcup_{i=1}^s S_i \subset Spec(k[x_1, \cdots, x_s]) $$ with $$S_i=Spec(k[x_1, x_1^{-1}, \cdots, x_{i-1}, x_{i-1}^{-1}, x_i, x_{i+1}, x_{i+1}^{-1}, \cdots, x_s, x_s^{-1}] $$ for $1\leq i\leq s$ by the structure of standard toric varieties (see Lemma 2.11 in \cite{CX}).

Since $div_{U}(x_i)=D_i$, one obtains that $\varsigma_{H_0}$ can be extended to $U_i\xrightarrow{{\varsigma_{H_0}}_i} S_i$ for $1\leq i\leq s$.  By gluing ${\varsigma_{H_0}}_i$ for $1\leq i\leq s$ together, one can extend $\varsigma_{H_0}$ to a morphism $U\rightarrow V$.

Applying Proposition \ref{fib-int} to the morphism of $H_0$-varieties
$$D_i \xrightarrow{{\varsigma_{H_0}}_i} div_V(x_i) \ \ \ \text{with} \ \ \ div_V(x_i) \cong T_0/\Bbb G_m \cong H_0/\varsigma_{H_0}^{-1}(\Bbb G_m) ,
 $$
 one concludes that each $\varsigma_{{H_0}_{i}}^{-1}$ is smooth with nonempty and geometrically integral fibres for $1\leq i\leq s$.
 Since $\codim(D_i,U)=\codim(div_V(x_i),V)=1$, one can conclude further that $\varsigma_{H_0}$ is flat. Thus  $\varsigma_{H_0}$ is smooth with nonempty and geometrically integral fibres.
\end{proof}

 \section{Proof of main Theorem \ref{main}}\label{proof}

We keep the same notation as that in the previous sections and assume $k$ is a number field in this section. We give a proof of Theorem \ref{main} by applying the results in previous section.  In particular,  $X$ is a smooth $G$-groupic variety over $k$ and $ Y\xrightarrow{\rho_X} X $ is a pull-back of a universal torsor of smooth compactification $X^c$ of $X$ under the torus $T$ over $k$ with $T^*=Pic(X_{\bar k}^c)$. Then $Y$ is an $H$-groupic variety where $H=\rho_{X}^{-1}(G)$ is a quasi-trivial linear algebraic group over $k$ by Lemma \ref{a_Y}. Moreover, one has
$$ H^{tor}  \cong T_0\times_k T_1 \ \ \ \text{with} \ \ \ T_0^*=\Div_{X_{\bar k} \setminus G_{\bar k}}(X_{\bar k}) \ \ \text{and} \ \ T_1^*=\Div_{X^c_{\bar k}\setminus X_{\bar k}}(X^c_{\bar k}). $$

Let $H\xrightarrow{\psi} T_1$ be the surjective homomorphism obtained by composing $\varsigma_H$ with the projection on $T_1$ and $H_0=ker(\psi)$. Then $\psi$ can be extended to a smooth morphism $Y\xrightarrow{\psi} T_1$ by Lemma \ref{ext}. Let $Y_0=\psi^{-1}(1_{T_1})$. It is a closed subscheme of $Y$.

\begin{lem} \label{pt} $X(k_v)=G(k_v) \cdot \rho_X(Y_0(k_v))$ for any $v\in \Omega_k$.
\end{lem}
\begin{proof} Since $Y\xrightarrow{\rho_X} X$ is a torsor under $T$, one has the following commutative diagram
$$ \begin{CD}
H(k_v) @>>> Y(k_v) \\
@V{\rho_G}VV  @VV{\rho_X}V \\
G(k_v) @>>> X(k_v) \\
@V{\partial_G}VV  @VV{\partial_X}V \\
H^1(k_v, T) @>{=}>> H^1(k_v,T)
\end{CD} $$
with the exact columns. Since $H^1(k_v, T)$ is finite by Theorem 6.14 in Chapter 6 in \cite{PR}, one obtains that both $\partial_X$ and $\partial_G$ are locally constant. Since $G(k_v)$ is dense in $X(k_v)$, one concludes that $ \partial_X(G(k_v))=\partial_X(X(k_v))$. This implies that for any $x_v\in X(k_v)$,  there exists $g_v\in G(k_v)$ such that $\partial_G(g_v)=\partial_X(x_v)$.
 For any $x_v\in X(k_v)$, $g_v\in G(k_v)$, one has $\partial(g_vx_v)=\partial(g_v)\partial(x_v)$, since this holds if $x_v\in G(k_v)$ and $G(k_v)$ is dense in $X(k_v)$.
 Therefore there is $y_v\in Y(k_v)$ such that $\rho_X(y_v)=g_v^{-1}x_v$. Thus $X(k_v)=G(k_v) \cdot \rho_X(Y(k_v))$.

Since $H(k_v)$ is dense in $Y(k_v)$, one has that $\psi(H(k_v))$ is dense in $\psi(Y(k_v))$. At the same time, $\psi(H(k_v))$ is an open subgroup of $T_1(k_v)$ by Proposition 3.3 in Chapter 3 in \cite{PR}. One concludes that $\psi(H(k_v))$ is closed and $\psi(H(k_v))=\psi(Y(k_v))$. For any $y\in Y(k_v)$, there is $h\in H(k_v)$ such that $\psi(y)=\psi(h)$. This implies that $h^{-1}y\in Y_0(k_v)$. Thus $Y(k_v)=H(k_v)\cdot Y_0(k_v)$. The result follows. \end{proof}

Let us now extend the statement of Proposition 4.2 in \cite{CX} on local approximation property for toric varieties to $G$-groupic varieties.

\begin{prop} \label{local} For any $x\in X(k_v)\setminus G(k_v)$, there is $y\in G(k_v)$ such that $y$ is as close to $x$ as required and $$\inv_v(\xi(x))=\inv_v(\xi(y))$$ for all $\xi\in \Br_1(X)$. \end{prop}

\begin{proof} By Lemma \ref{pt}, there is $g\in G(k_v)$ and $y_0\in Y_0(k_v)$ such that $x=g\cdot \rho_X(y_0)$. Let $M$ be an open neighbourhood of $x$ in $X(k_v)$. Then $y_0 \in \rho_X^{-1}(g^{-1}M)\cap Y_0(k_v)$ is a non-empty open subset of $Y_0(k_v)$. Since $H_0(k_v)$ is dense in $Y_0(k_v)$, there is $h_0 \in H_0(k_v) \cap \rho_X^{-1}(g^{-1}M) $.

Let $$y=g\cdot \rho_X(h_0)\in G(k_v)\cap M . $$

For any $\xi \in Br_1(X)$, one has
$$ \inv_v(\xi(y)) = \inv_v(\xi(g\cdot \rho_X(h_0)))=\inv_v (g^*(\xi)(\rho_X(h_0)))= \inv_v (\rho_X^*(g^*(\xi))(h_0))$$
and
$$ \inv_v(\xi(x)) = \inv_v(\xi(g\cdot \rho_X(y_0)))=\inv_v (g^*(\xi)(\rho_X(y_0)))= \inv_v (\rho_X^*(g^*(\xi))(y_0)) . $$  Since $\Br_1(Y_0)$ is constant by Proposition \ref{y_0}, one obtains the desired result. \end{proof}

Let $S$ be a finite subset of $\Omega_k$ containing all archimedean places such that the following conditions hold:

i) The open immersion $i_G:  G\hookrightarrow X$ extends to an open immersion $i_{\mathcal{G}}: \mathcal{G} \hookrightarrow \mathcal{X}$ where $\mathcal{G}$ is a smooth group scheme over $O_{k,S}$ with $\mathcal{G}\times_{O_{k,S}}k=G$ and $\mathcal{X}$ is a smooth scheme over $O_{k,S}$ with $\mathcal{X}\times_{O_{k,S}}k=X$. Moreover, the action $G\times_k X\xrightarrow{a_X} X$ extends to an action $\mathcal{G}\times_{O_{k,S}}\mathcal{X}\xrightarrow{a_{\mathcal{X}}} \mathcal{X}$ which is compatible with the multiplication of $\mathcal{G}$.

ii) The torsor $Y\xrightarrow{\rho_X} X$ under $T$ extends to a torsor $\mathcal{Y} \xrightarrow{\rho_{\mathcal{X}}} \mathcal{X}$ under a smooth connected commutative group scheme $\mathcal{T}$ over $O_{k,S}$, where $\mathcal{Y}$ and $\mathcal{T}$ are smooth over $O_{k,S}$ such that $\mathcal{Y}\times_{O_{k,S}}k=Y$ and $\mathcal{T}\times_{O_{k,S}}k=T$. Moreover, the surjective homomorphism $H\xrightarrow{\rho_G} G$ extends to a surjective smooth homomorphism of smooth group schemes $\mathcal{H} \xrightarrow{\rho_{\mathcal{G}}} \mathcal{G}$ over $O_{k,S}$ with $\mathcal{H}\times_{O_{k,S}}k=H$.

iii) The surjective homomorphism $H\xrightarrow{\psi} T_1$ extends to a smooth surjective homomorphism of smooth group schemes $\mathcal{H}\xrightarrow{\psi} \mathcal{T}_1$ over $O_{k,S}$ such that $\mathcal{H}_0=\psi^{-1}(1_{\mathcal{T}_1})$ is a connected smooth group scheme over $O_{k,S}$ by Lemma \ref{h_0} with $\mathcal{T}_1\times_{O_{k,S}}k=T_1$ and $\mathcal{H}_0\times_{O_{k,S}}k=H_0$. The extension of $\psi$ in Lemma \ref{ext} extends to a smooth morphism $\mathcal{Y}\xrightarrow{\psi} \mathcal{T}_1$ such that $\mathcal{Y}_0=\psi^{-1}(1_{\mathcal{T}_1})$ with $\mathcal{Y}_0\times_{O_{k,S}}k=Y_0$.

iv) The open immersion $i_H:  H\hookrightarrow Y$ extends to an open immersion $i_{\mathcal{H}}: \mathcal{H} \hookrightarrow \mathcal{Y}$ such that the action $H\times_k Y\xrightarrow{a_Y} Y$ in Lemma \ref{a_Y} extends to an action $\mathcal{H}\times_{O_{k,S}}\mathcal{Y}\xrightarrow{a_{\mathcal{Y}}} \mathcal{Y}$ which is compatible with the multiplication of $\mathcal{H}$.

\begin{lem}\label{loc-int} For any $v\not\in S$, one has
$$ G(k_v) \cap \mathcal{X}(O_v) = \mathcal{G}(O_v) \cdot \rho_{\mathcal{X}}(H_0(k_v)\cap \mathcal{Y}_0(O_v)) $$
\end{lem}
\begin{proof} Since $H^1_{et}(O_v, \mathcal{T})=0$, one has $\rho_{\mathcal{X}}(\mathcal{Y}(O_v))=\mathcal{X}(O_v)$. For any $x\in G(k_v) \cap \mathcal{X}(O_v)$, there is $y\in H(k_v)\cap \mathcal{Y}(O_v)$ such that $x=\rho_{\mathcal{X}}(y)$ because $H$ is the pull-back of $Y$ over $G$. Since $H^1_{et}(O_v, \mathcal{H}_0)=0$, there is $h\in \mathcal{H}(O_v)$ such that $\psi(h)=\psi(y)$. This implies that $h^{-1}y\in H_0(k_v)\cap \mathcal{Y}_0(O_v)$. Therefore
$$ x = \rho_{\mathcal{X}}(h) \cdot \rho_{\mathcal{X}}(h^{-1} y) \in \mathcal{G}(O_v) \cdot \rho_{\mathcal{X}}(H_0(k_v)\cap \mathcal{Y}_0(O_v)) $$ as required.
\end{proof}

For a connected linear algebraic group $G$, one defines the set
$$ \Sha^1(k, G)= \ker (H^1(k, G) \rightarrow \prod_{v\in \Omega_k} H^1(k_v, G)) . $$

The statement below gathers classical theorems on the Hasse Principle.

\begin{thm} \label{sha} Let $G_1$ be a connected quasi-trivial linear algebraic group over a number field $k$. Then

(1) One has $H^1(k_v, G_1)=\{1\}$ for any no real prime $v$ of $k$;

(2) One has that $G_1$ satisfies weak approximation; 

(3) One has $H^1(k, G_1)\cong \prod_{v \ \text{real}} H^1(k_v, G_1)$, and then $\Sha^1(k, G_1)=\{1\} $.  \end{thm}

\begin{proof} The results of (1) and (3) follow from applying Proposition 9.2 in \cite{CT08} to $G_1^{red}$ and Lemma 1.13 in \cite{Sansuc} (see also Lemma 7.3 in \cite{GM}). 

For (2), one has the following digram of short exact sequences
\[ \begin{CD}
 R_u(G_1)(k) @>>> G_1(k) @>>> G_1^{red}(k) @>>>  H^1(k, R_u(G_1))=1  \\
@VVV @VVV @VVV @VVV \\
\prod_{v} R_u(G_1)(k_v) @>>> \prod_{v} G_1(k_v) @>>> \prod_{v} G_1^{red}(k_v) @>>> \prod_{v}H^1(k_v, R_u(G_1))=1.
 \end{CD} \]
Since $G_1^{red}$ satisfies weak approximation by Proposition 9.2 in \cite{CT08}, the result follows from weak approximation of $R_u(G_1)$ and tracing the above digram.
\end{proof}

In \cite{CX} Proposition 3.4, using Harari's work in \cite{Ha08}, we proved a relative strong approximation for tori. Using Demarche's  work in \cite{D} and the above construction, we now give a relative strong approximation theorem for arbitrary connected linear algebraic groups.

\begin{prop} \label{rel-sa} Let $G_1\xrightarrow{f} G_2$ be a homomorphism of connected linear algebraic groups over $k$ and $ \Br_a(G_2)\xrightarrow{f^*} \Br_a(G_1)$ the induced map. Suppose $G_i'(k_\infty)$ is not compact for any non-trivial simple factor $G_i'$ of $G_i^{sc}$ for $i=1,2$. If $\Sha^1(k,G_1)=\{1\}$,
then for any open subset $W$ of $G_2({\mathbf A}_k)_{\bullet}$ with $$W\cap G_2({\mathbf A}_k)_{\bullet}^{\ker f^*} \neq \emptyset ,$$ there are $x\in G_2(k)$ and $y\in G_1({\mathbf A}_k)_{\bullet}$ such that $x f(y) \in W$.
\end{prop}
\begin{proof} Let $R_u(G_i)$ be the unipotent radical of $G_i$ and $G_i^{red}$ be the reductive part of $G_i$ for $i=1,2$. Since $H^1(\Bbb R, R_u(G_i))=1$, one has
$$ 1\rightarrow R_u(G_i)(\Bbb R) \rightarrow G_i(\Bbb R) \rightarrow G_i^{red}(\Bbb R) \rightarrow 1  $$ for $i=1,2$. This induces an isomorphism
$$\pi_0(G_i^{red}(\Bbb R))\cong \pi_0 (G_i(\Bbb R))$$ by connectedness of $R_u(G_i)(\Bbb R)$ for $i=1,2$.

Let $G_i^{sc}$ be a semi-simple simply connected covering of the semi-simple part $G_i^{ss}=[G_i^{red}, G_i^{red}]$ of $G_i$ for $i=1,2$. Then $G_i^{sc}(\Bbb R)$ is connected for $i=1,2$ by Proposition 7.6 of Chapter 7 in \cite{PR}. This implies that $G_i^{scu}(\Bbb R)$ in \cite{D} Corollary 3.20 is connected.

By Demarche's work \cite{D}  Corollary 3.20 for $S_0=\infty_k$, there are compatible exact sequences:
\[ \begin{CD}
1 @>>> \overline{G_1(k)} @>>> G_1({\mathbf A}_k)_{\bullet} @>>> \Br_a(G_1)^D @>>> \Sha^1(k,G_1)=\{1\}  \\
@.  @VVV @VV{f}V @VV{(f^*)^D}V @VVV  \\
1 @>>> \overline{G_2(k)} @>>> G_2({\mathbf A}_k)_{\bullet} @>>> \Br_a(G_2)^D @>>> \Sha^1(k,G_2) @>>> 1
 \end{CD} \]
 where $(-)^D:=\hom(-,\Bbb Q/\Bbb Z)$.
 There exists $y\in G_1({\mathbf A}_k)_{\bullet}$ such that $$\alpha f(y)^{-1} \in \overline{G_2(k)} \ \ \text{ for} \ \  \alpha \in  W\cap G_2({\mathbf A}_k)_{\bullet}^{\ker f^*} $$ by the above diagram. Since $W f(y)^{-1}$ is an open subset containing $\alpha f(y)^{-1}$, one concludes that there is $x\in G_2(k)$ such that $xf(y)\in W$ as required. \end{proof}

The following proposition benefits from discussion with J.-L. Colliot-Th\'el\`ene and Dasheng Wei which refines Proposition 3.6 in \cite{CX}. 

\begin{prop} \label{sign} Suppose $U$ is an open subset of $n$-dimensional affine space $\Bbb A^n$ over $k$ such that  $\codim(\Bbb A^n\setminus U, \Bbb A^n)\geq 2$ and $v_0$ is a real prime of $k$. Let
$$ U_{v_0}^+=U(k_{v_0})\cap \{ (x_1, \cdots, x_n)\in k_{v_0}^n : \ x_i >0 \ \text{in} \ k_{v_0} \ \text{with} \ 1\leq i\leq n\} $$
where $k_{v_0}^n=\Bbb A^n(k_{v_0})$ by fixing the coordinates over $k$. If $W$ is a non-empty open subset of $U({\mathbf A}_k^{v_0})$ where ${\mathbf A}_k^{v_0}$ is the adeles of $k$ without $v_0$-component,  then
$$ U(k) \cap (U_{v_0}^+ \times  W) \neq \emptyset .  $$
\end{prop}

\begin{proof} 
Without loss of generality, one can assume that $W=\prod_{v\neq v_0} W_v $.  By using the fixed coordinates, we consider the projection to the first coordinate
$$p:  \Bbb A^n \rightarrow \Bbb A^1; \ \ \ (x_1, \cdots, x_n) \mapsto x_1. $$ It is clear that $p^{-1}(x) \cong \Bbb A^{n-1}$ over $k$ for any $x\in k$. 
Since $p^{-1}(x)(k_v)$ is Zariski dense in $p^{-1}(x)$ (see Theorem 2.2 in Chapter 2 of \cite{PR}) and $\dim(p^{-1}(x))> \dim(p^{-1}(x)\cap Z)$ for any $x\in \Bbb A^1(k_v)$,  one has
  $$ (p^{-1}(x)\cap U)(k_v) = p^{-1}(x)(k_v) \setminus (p^{-1}(x)\cap Z)(k_v) \neq \emptyset $$
  for any $v\in \Omega_k$. Since $p$ is smooth, one concludes that $p(U_{v_0}^+)\times \prod_{v\neq v_0} p(W_v)$ is an non-empty open subset of ${\bf A}_k$ and $p(U_{v_0}^+)=\Bbb R^+$ where $\Bbb R^+$ is the set of all positive real numbers.
  
When $k=\Bbb Q$, one has $$\Bbb Q\cap [\Bbb R^+\times  \prod_{v<\infty} p(W_v) ]\neq \emptyset$$ by Dirichlet's prime number theorem. 

Otherwise,  there is $\epsilon \in O_{k}^\times$ such that 
$ \epsilon >1$ at $v_0$ and $ |\epsilon|_{v}<1 $ for all $v\in \infty_k\setminus \{v_0\}$ by 33:8 in \cite{OM}. Let $\Sigma$ be a finite subset of $\Omega_k$ containing $\infty_k$ such that $p(W_v)=O_v$ for all $v \not \in \Sigma$ and  $\beta_v\in O_k$ such that $ord_v(\beta_v)>0$ and $ord_{w}(\beta_v)=0$ for $w\neq v$ and $w<\infty_k$ by finiteness of class number of $O_k$ for each $v< \infty_k$.  By strong approximation of ${\Bbb A}^1$, there is $a\in k$ such that  $a\in k_{v_0}\times \prod_{v\neq v_0} p(W_v)$. Let $l_v$ be a sufficiently large integer such that 
$ a +  \beta_v^{l_{v}}O_v \in p(W_v) $ for each $v\in \Sigma\setminus \infty_k$. Let $N$ be a sufficiently large integer such that
$$b=a+ \epsilon^{N} \prod_{v\in \Sigma \setminus \infty_k} \beta_v^{l_v} \in p(W_v) $$ for all $v\in \infty_k\setminus \{v_0\}$ and  is positive at $v_0$. Therefore 
$$ b\in k\cap [\Bbb R^+\times  \prod_{v\neq v_0} p(W_v)] \neq \emptyset . $$

If $\dim( p^{-1}(x)\cap Z)= \dim(Z)$, then there is a generic point $y$ of irreducible components of $Z$ such that $p(y)=x$. Since the irreducible components of $Z$ is finite, one has 
  $$ \{ x\in k:  \ \dim( p^{-1}(x)\cap Z)= \dim(Z) \} $$ is finite. There is $$y\in k \cap [\Bbb R^+\times  \prod_{v\neq v_0} p(W_v)] \ \ \text{ such that} \ \  \codim(p^{-1}(y)\cap Z, p^{-1}(y))\geq 2 . $$
  By induction on $U\cap p^{-1}(y)$, one gets 
$$ z\in (U\cap p^{-1}(y))(k) \cap [ (p^{-1}(y)(k_{v_0})\cap  U_{v_0}^+) \times \prod_{v\neq v_0} (p^{-1}(y)(k_v)\cap W_v)].$$  
Combing $z$ and $y$, one concludes $U(k) \cap (U_{v_0}^+ \times  W) \neq \emptyset$ .  
\end{proof}

\begin{cor}\label{str-ref} Let $(T_0\hookrightarrow V)$ be a standard toric variety over a number field $k$, $v_0\in \infty_k$ and ${\bf A}_k^{v_0}$ be the adeles of $k$ without $v_0$-component. If $W$ is a non-empty open subset of $V({\mathbf A}_k^{v_0})$ and $W_{v_0}$ is a connected component of $T_0(k_{v_0})$, then
$$ V(k) \cap (W_{v_0}\times  W) \neq \emptyset .  $$
\end{cor}

\begin{proof} We first prove the result is true if $W_{v_0}$ is the connected component of identity of $T_0(k_{v_0})$ as an $\Bbb R$-Lie group. If $v_0$ is complex, then $W_{v_0}=T_0(k_{v_0})$. Since $T_0(k_{v_0})$ is dense $V(k_{v_0})$, one has 
$ V(k) \cap (W_{v_0}\times  W) \neq \emptyset $ by Lemma 2.11 in \cite{CX} and Corollary 3.7 for $S=\{ v_0 \}$ in \cite{CX}. 

Otherwise, $v_0$ is real. Since $$T_0=\Res_{K_1/k}(\Bbb G_m) \times \cdots \times \Res_{K_s/k}(\Bbb G_m)$$ where $K_i=k(\theta_i)$ is a finite extension of $k$ with $d_i=[K_i : k]$ for $1\leq i\leq s$, one can choose $\theta_i$ such that $\theta_i$ is positive in $(K_i)_w$  for all real primes $w$'s of $K_i$ above $v_0$ and the real part $\Re(\theta_i^j)$ of $\theta_i^j$ in $(K_i)_w$ is positive for $1\leq j\leq d_i-1$ for all complex primes $w$'s of $K_i$ above $v_0$ for $1\leq i\leq s$. Fixing an isomorphism $\Res_{K_i/k}(\Bbb A^1) \rightarrow \Bbb A_k^{d_i}$ such that 
$$ \Res_{K_i/k}(\Bbb A^1)(A)=\sum_{j=0}^{d_i-1} A\theta_i^j \rightarrow A^{d_i}; \ \ \ \ \sum_{j=0}^{d_i-1} a_j \theta_i^j \mapsto (a_0, \cdots, a_{d_i-1}) $$ with $a_0, \cdots, a_{d_i-1}\in A$ for any $k$-algebra $A$ and $1\leq i\leq s$. By choosing such coordinates, one has  $T_0\subset V \subset \Bbb A^d$ over $k$ with $d=\sum_{i=1}^s d_i$ and $$T_0(k_{v_0})^+\supset \{(x_1,\cdots, x_d)\in k_{v_0}^d: \ x_i>0 \ \text{in $k_{v_0}$ for $1\leq i\leq d$} \} $$  where $T_0(k_{v_0})^+$ is the connected component of identity. One concludes $ V(k) \cap (W_\infty\times  W) \neq \emptyset $ 
by Proposition \ref{sign}. 

In general, since $T_0(k)$ is dense in $T_0(k_\infty)$, there is $t\in T_0(k)$ such that $t\cdot W_\infty =T_0(k_\infty)^+$. The result follows from applying the above result to open set $T_0(k_\infty)^+ \times t\cdot W$.  \end{proof}

We can prove strong approximation for $Y_0$ by using strong approximation with Brauer-Manin obstruction for $H_0$.

\begin{prop}\label{add} Under the assumptions of Theorem \ref{main}, the variety $Y_0$ satisfies strong approximation off $\infty_k$. \end{prop}

\begin{proof} By Proposition \ref{chi}, there exists a standard toric variety $(T_0\hookrightarrow V)$ and an open $H_0$-subvariety $U\subset Y_0$ such that $H_0\subset U$, $\codim(Y_0\setminus U,Y_0)\geq 2$ and the morphism $H_0\xrightarrow{\varsigma_{H_0}}T_0$ can be extended to a smooth morphism $U\xrightarrow{\varsigma_U} V$  with nonempty and geometrically integral fibres. Then one only needs to show that $U$ satisfies strong approximation off $\infty_k$.

Let $W$ be a non-empty open subset of $U(\mathbf A_k)_{\bullet}$. Since $\varsigma_U$ is smooth and all fibers of $\varsigma_U$ are not empty and geometrically integral, one obtains that $\varsigma_U(W)$ is a non-empty open subset of $V(\Bbb A_k)_{\bullet}$. Applying Corollary \ref{str-ref}, one gets $x\in T_0(k)$ such that  $$\varsigma_{U}^{-1}(x) \cap W \neq \emptyset  . $$ Since $\varsigma_{U}^{-1}(x)$ is smooth and $\varsigma_{H_0}^{-1}(x) = \varsigma_{U}^{-1}(x) \cap H$ is an open dense subset of $\varsigma_{U}^{-1}(x)$, one further has
$$\varsigma_{H_0}^{-1}(x) \cap W \neq \emptyset  . $$

Since $Pic((T_0)_{\bar k})=Pic((H_0)_{\bar k})=0$, the quotient map $H_0\xrightarrow{\varsigma_{H_0}} T_0$ induces an isomorphism $$\Br_1(T_0)\cong\Br_1(H_0)$$ by Lemma \ref{h_0} above and Lemma 2.1 in \cite{CTX}.
Therefore
$$ (W\cap H_0({\bf A}_k)_{\bullet})^{\Br_1(H_0)} \supseteq (W\cap \varsigma_{H_0}^{-1}(x))^{\Br_1(H_0)} = W\cap \varsigma_{H_0}^{-1}(x) \neq \emptyset $$ by the functoriality of Brauer-Manin pairing.

By Lemma \ref{h_0}, one has $H_0^{sc}\xrightarrow{\cong}G^{sc}$.
Therefore one can apply Corollary 3.20 in Demarche \cite{D} to $H_0$ and obtains $H_0(k)\cap W\neq \emptyset .$ \end{proof}

\begin{proof} (\emph{Proof of Theorem \ref{main}.})

Let $W=\prod_{v\in \Omega_k} W_v$ be an open subset of $X(\mathbf A_k)_{\bullet}$ such that there exist
 $$  (x_v)_{v\in \Omega_k} \in W\cap X(\mathbf A_k)_{\bullet}^{\Br_1 X}  $$
and a sufficiently large finite subset $S_1$ of $\Omega_k$ containing $S$ with $W_v={\mathcal X}(O_v)$ for all $v\not\in S_1$. By Proposition \ref{local}, one can assume that $x_v\in G(k_v)$ for all $v\in \Omega_k$. Then  $$x_v\in W_v \cap G(k_v)={\mathcal X}(O_v)\cap G(k_v) ={\mathcal G}(O_v)\cdot \rho_{\mathcal{X}}(H_0(k_v)\cap \mathcal{Y}_0(O_v))$$ for $v\not\in S_1$ by Lemma \ref{loc-int}. Let $$g_v\in {\mathcal G}(O_v) \ \ \ \text{and} \ \ \ \beta_v\in {\mathcal Y}_0(O_v)\cap H_0(k_v)$$ such that $x_v=g_v \cdot \rho_{\mathcal X}(\beta_v)$ for all $v\not\in S_1$ and $g_v=x_v$ for $v\in S_1$. Then $(g_v)_{v\in \Omega_k}\in G(\mathbf A_k)$.

By Proposition \ref{y_0}, one has
$$ \inv_v(\xi(x_v))= \inv_v(\xi(g_v\cdot \rho_X(\beta_v)))=\inv_v((\rho_X^* g_v^* \xi)(\beta_v))=\inv_v((\rho_X^* g_v^* \xi)(1_{H_0}))= \inv_v(\xi(g_v))$$  for all $\xi\in \Br_1(X)$ and $v\in \Omega_k$.
Since $H_0$ is quasi-trivial by Lemma \ref{h_0}, the set $\Sha^1(k, H_0)=\{1\}$ by Theorem \ref{sha}.
By Lemma \ref{h_0},  for any non-trivial simple factor $H'$ of $H_0^{sc}$, $H'(k_\infty)$ is not compact.

Applying Proposition \ref{rel-sa} to $H_0\rightarrow G$, and using Proposition \ref{br},  there exist $g\in G(k)$ and $y_\mathbf A\in H_0(\mathbf A_k)_{\bullet}$ such that $$g\rho_X (y_\mathbf A)\in (\prod_{v\in \infty_k} i_{G}^{-1}(W_v) \times \prod_{v\in S_1\setminus \infty_k} (W_v\cap G(k_v)) \times \prod_{v\not\in S_1} {\mathcal G}(O_v)) $$ where the open immersion $i_G: G\hookrightarrow X$ induces $i_G: \pi_0(G(k_v)) \rightarrow \pi_0(X(k_v))$ for $v\in \infty_k$. This implies that $ y_{\mathbf A}$ is in the open subset $\rho_{X}^{-1}(g^{-1} \cdot W)$ of $Y_0(\mathbf A_k)_{\bullet}$.
By Proposition \ref{add}, one concludes that there is $$y\in Y_0(k) \cap \rho_{X}^{-1}(g^{-1} \cdot W)$$ and this is equivalent to that $g \cdot \rho_{X} (y) \in  W $ as desired. \end{proof}

\section{Approximation at archimedean primes}\label{ss}

For classical strong approximation, there is no difference between the archimedean primes and the non-archimedean primes. In this section, we discuss strong approximation off $any$ finite non-empty subset $S$ of $\Omega_k$.  We keep the same notation as that in the previous sections and assume $k$ is a number field in this section.

First one needs to modify Definition \ref{sa} as follows.  

\begin{definition} \label{msa} Let $S$ be a non-empty finite subset of $\Omega_k$ for a number field $k$ and $X$ be a variety over $k$ and $pr^S$ be the projection $X({\bf A}_k)\rightarrow X({\bf A}_k^S)$ where ${\bf A}_k^S$ is the adeles of $k$ without $S$-components.

(1) If $X(k)$ is dense in $pr^S(X({\mathbf A}_k))$, we say $X$ satisfies strong approximation off $S$.

(2) If $X(k)$ is dense in $ pr^S(X({\mathbf A}_k)^{B})$ for some subset $B$ of $\Br(X)$, we say $X$ satisfies strong approximation with respect to $B$ off $S$.
\end{definition}

One can refine Proposition \ref{add} to adapt for any finite subset $S$ by applying the fibration method in \cite{CTX1}.

\begin{prop}\label{refine-add} Let $Y_0$ be a variety given by Proposition \ref{y_0}  over a number field $k$ and $S$ be a non-empty finite subset of $\Omega_k$.  If $\prod_{v\in S} G'(k_v)$ is not compact for any non-trivial simple factor $G'$ of $G^{sc}$, then $Y_0$ satisfies strong approximation off $S$. 
\end{prop}

\begin{proof} 

By Proposition \ref{chi}, there exists a standard toric variety $(T_0\hookrightarrow V)$ and an open $H_0$-subvariety $U\subset Y_0$ such that $H_0\subset U$, $\codim(Y_0\setminus U,Y_0)\geq 2$ and the morphism $H_0\xrightarrow{\varsigma_{H_0}}T_0$ can be extended to a smooth morphism $U\xrightarrow{\varsigma_U} V$  with non-empty and geometrically integral fibres. Then one only needs to show that $U$ satisfies strong approximation off $S$. 

One can verify the condition (i), (ii) and (iii) of Proposition 3.1 in \cite{CTX1} for the fibration $U\xrightarrow{\varsigma_U} V$ with the open subset $T_0$ of $V$. 

For condition (i), we have $V$ satisfies strong approximation off $S$ by Corollary 3.7 in \cite{CX}.  If $S$ contains a real prime $v_0$, we will apply the stronger version of strong approximation off $S$ for $V$ by Corollary \ref{str-ref}. 

For condition (ii), we have $H_0^{ss}\cong G^{sc}$ which is semi-simple and simply connected by Lemma \ref{h_0}. Therefore $\prod_{v\in S} H'(k_v)$ is not compact for any non-trivial simple factor $H'$ of $H_0^{ss}$ by the assumption. Since $R_u(\ker(\varsigma_{H_0}))$ is an affine space,  we have that $H_0^{ss}$ and  $R_u(\ker(\varsigma_{H_0}))$ satisfy strong approximation off $S$ by Theorem 7.12 of Chapter 7 in \cite{PR}. Since $$H^1(k_v, R_u(\ker(\varsigma_{H_0}))=\{1\} $$ for each $v\in S$,   the quotient maps $\ker(\varsigma_{H_0})(k_v) \rightarrow H_0^{ss}(k_v)$ is surjective for each $v\in S$. By Proposition 3.1 in \cite{CTX1}, one concludes that $\ker(\varsigma_{H_0})$ satisfies strong approximation off $S$.

For any $t_0\in T_0(k)$, one knows that the fiber $\varsigma_{H_0}^{-1}(t_0)$ is a $\ker(\varsigma_{H_0})$-torsor over $k$. By Lemma \ref{h_0} and Theorem \ref{sha},  one concludes $\Sha^1(\ker(\varsigma_{H_0}))=\{1\}$.   If $\varsigma_{H_0}^{-1}(t_0)({\bf A}_k)\neq \emptyset, $  then $\ker(\varsigma_{H_0})\cong \varsigma_{H_0}^{-1}(t_0)$ over $k$. Therefore $\varsigma_{H_0}^{-1}(t_0)$ satisfies strong approximation off $S$.

For condition (iii), for any no real prime $v\in S$,  we have $H^1(k_v, \ker(\varsigma_{H_0}))=\{1\}$ by Lemma \ref{h_0} and Theorem \ref{sha}. Therefore $H_0(k_v)\xrightarrow{\varsigma_{H_0}} T_0(k_v)$ is surjective. We only need to consider $S$ contains a real prime $v_0$. Since $H^1(k_{v_0}, \ker(\varsigma_{H_0}))$ is finite by Theorem 6.14 in Chapter 6 of \cite{PR}, one has  $\varsigma_{H_0}(H(k_{v_0}))\supseteq T_0(k_{v_0})^+$ where $T_0(k_{v_0})^+$ is the connected component of identity. We apply the stronger version of strong approximation off $S$ for $V$ by Corollary \ref{str-ref} and obtain $t_0\in T_0(k)$ and $\varsigma_{H_0}^{-1}(t_0)(k_{v})\neq \emptyset$ for all $v\in \Omega_k$. The rest of argument follows from the same as those of Proposition 3.1 in \cite{CTX1}.  \end{proof}

The following lemma provides the computation of Brauer-Manin invariants of points with with group action for algebraic parts.  

\begin{lem} \label{inv+act} Let $G_1$ be a connected linear algebraic group over $k$ and $P$ be a smooth variety with an action
$$ a_P: G_1\times_k P \rightarrow  P$$ over $k$. Suppose $P(k)\neq \emptyset$ and fix $\nu\in P(k)$. Then one has 
$$ \sum_{v\in \Omega_k} \inv_v(\alpha((g_v)\cdot (x_v))) = \sum_{v\in \Omega_k} \inv_v (\iota_{G_1}^*(\alpha)(g_v)) + \sum_{v\in \Omega_k} \inv_v(\alpha (x_v)) $$ for any $(g_v)\in G_1 ({\bf A}_k)$ and $(x_v)\in P({\bf A}_k)$ and $\alpha\in \Br_1(P)$, where 
$$ \iota_{G_1}: G_1 \xrightarrow{id\times \nu} G_1\times_k P \xrightarrow{a_P} P  . $$
\end{lem}
\begin{proof}
By the functoriality of Brauer-Manin pairing, one has 
$$\sum_{v\in \Omega_k} \inv_v(\alpha((g_v)\cdot (x_v))) =\sum_{v\in \Omega_k} \inv_v (a_P^*(\alpha)(g_v, x_v)) . $$ Since both $P$ and $G_1$ have rational points, one has 
$$ \Br_1(G_1\times_k P) =p_{G_1}^* (\Br_e(G_1) \oplus p_P^*(\Br_1(P)) $$ by Sansuc's exact sequence (see (6.10.3) of Proposition 6.10 in \cite{Sansuc}), where $p_{G_1}$ and $p_P$ are the projection of $G_1\times_k P$ to $G_1$ and $P$ respectively and $\Br_1(G_1)=\Br_e(G_1)\oplus \Br(k)$ by using the section $1_G$. The result follows from the functoriality of Brauer-Manin pairing. 
\end{proof}

In order to establish strong approximation off any finite non-empty subset $S$, we need the following descent result other than Proposition \ref{rel-sa}. 

\begin{prop}\label{ssd}
Let $G_1$ be a connected linear algebraic group with $\Sha^1(k, G_1)=\{1\}$ and $P=G_1/M$ where $M$ is a group of multiplicative type over $k$. Let $\pi: G_1\rightarrow P$ be the quotient map. Then 
$$ P({\bf A}_{k})^{\ker(\pi^*)} = \pi (G_1({\bf A}_k)) \cdot P(k) $$ where $\pi^*: \Br_1(P)\rightarrow \Br_a (G_1)$ induced by $\pi$. 
\end{prop}
\begin{proof} When $M$ is a group of multiplicative type, there is a map 
$$ \theta(G_1):  \ H^1(k, M^*) \rightarrow \Br_1(P)$$ obtained by restricting the cup product $H^1(k, M^*)\times H^1_{et}(P, M) \rightarrow \Br_1(P)$ to $G$ (see line -3 of P.316 in \cite{CTX}). Since the following digram of the cup product
\[ \begin{CD}
H^1(k, M^*)\times H^1_{et}(P, M)  @>>> \Br_1(P)  \\
 @V{id\times \pi^*}VV  @VV{\pi^*}V  \\
 H^1(k, M^*) \times H_{et}^1(G_1, M) @>>> \Br_1(G_1)
  \end{CD} \] commutes and $G_1$ as a torsor over $P$ under $M$ becomes a trivial torsor over $G_1$ under $M$, one concludes that $$\im(\theta(G_1))\subseteq \ker(\pi^*) . $$

By Proposition 2.7 in \cite{CTX}, one has the following commutative diagram
\[ \begin{CD}
G_1(k) @>>> P(k) @>>> H^1 (k, M) @>>> H^1(k,G_1)  \\
 @VVV  @VVV @VVV @VVV   \\
 G_1({\bf A}_k) @>>> P({\bf A}_k) @>>> \prod_{v}' H^1(k_v, M) @>>> \prod_v' H^1(k_v, G_1) \\
 @. @VVV  @VVV @. \\
@.  \Br_1(P)^D @>{\theta(G_1)^D}>> H^1(k, M^*)^D  @. 
  \end{CD} \]
where $(-)^D:=\hom(-,\Bbb Q/\Bbb Z)$ such that the rows are exact by Proposition 36 in \S 5.4, Chapter I  of \cite{Ser} and the third column is exact by Theorem 6.3 in \cite{D0}. Since $\Sha^1(k, G_1)=\{1\}$,  one concludes that 
$$ P({\bf A}_k)^{Im(\theta(G_1))} = \pi (G_1({\bf A}_k)) \cdot P(k) $$ by the above diagram and Corollary 1 of P.50 in \cite{Ser}.
This implies that 
$$  \pi (G_1({\bf A}_k)) \cdot P(k) \subseteq  P({\bf A}_k)^{\ker(\pi^*)}\subseteq P({\bf A}_k)^{Im(\theta(G_1))} = \pi (G_1({\bf A}_k)) \cdot P(k) $$  as required by Lemma \ref{inv+act}.   
\end{proof}

The main result of this section is the following theorem.

\begin{thm} \label{arch-main} Let $X$ be a smooth $G$-groupic variety over a number field $k$ and $S$ be a non-empty finite subset of $\Omega_k$ of $k$. If $\bar{k}[X]^\times =\bar k^\times$ and $\prod_{v\in S} G'(k_v)$ is not compact for any non-trivial simple factor $G'$ of $G^{sc}$, then $X$ satisfies strong approximation with respect to $\Br_1(X)$ off $S$.
\end{thm}

\begin{proof} Let $W=\prod_{v\in S}X(k_v) \times \prod_{v\not\in S} W_v$ be an open subset of $X(\mathbf A_k)$ such that  $$  (x_v)_{v\in \Omega_k} \in W\cap X(\mathbf A_k)^{\Br_1 X} . $$  Then there exists a sufficiently large finite subset $S_1$ of $\Omega_k$ containing $S\cup \infty_k$ such that i), ii), iii) and iv) before Lemma \ref{loc-int} holds.  By Proposition \ref{local}, one can assume that $x_v\in G(k_v)$ for all $v\in \Omega_k$. Then  $$x_v\in W_v \cap G(k_v)={\mathcal X}(O_v)\cap G(k_v) ={\mathcal G}(O_v)\cdot \rho_{\mathcal{X}}(H_0(k_v)\cap \mathcal{Y}_0(O_v))$$ for $v\not\in S_1$ by Lemma \ref{loc-int}. There are $$g_v\in {\mathcal G}(O_v) \ \ \ \text{and} \ \ \ \beta_v\in {\mathcal Y}_0(O_v)\cap H_0(k_v)$$ such that $x_v=g_v \cdot \rho_{\mathcal X}(\beta_v)$ for all $v\not\in S_1$ and $g_v=x_v$ for $v\in S_1$. Then $(g_v)_{v\in \Omega_k}\in G(\mathbf A_k)$. By Proposition \ref{y_0}, one has
$$ \inv_v(\xi(x_v))= \inv_v(\xi(g_v\cdot \rho_X(\beta_v)))=\inv_v((\rho_X^* g_v^* \xi)(\beta_v))=\inv_v((\rho_X^* g_v^* \xi)(1_{H_0}))= \inv_v(\xi(g_v))$$  for all $\xi\in \Br_1(X)$ and $v\in \Omega_k$. This implies that $$(g_v)_{v\in \Omega_k}\in W \cap G(\mathbf A_k)^{\ker(\rho_G^*)}$$ by Proposition \ref{br}.  

By Lemma \ref{constant}, we know $H_0\xrightarrow{\rho_G} G$ is surjective and its kernel is a group of multiplicative type. By Lemma \ref{h_0} and Theorem \ref{sha}, $\Sha^1(k, H_0)=\{1\}$. 
Applying Proposition \ref{ssd} to the quotient map $H_0\xrightarrow{\rho_G} G$, one has $g\in G(k)$ and $y_\mathbf A\in H_0(\mathbf A_k)$ such that $$g\rho_X (y_\mathbf A)=(g_v)_{v\in\Omega_k}  \in W \ \ \ \text{and} \ \ \ 
y_{\mathbf A} \in  [\prod_{v\in S} Y_0(k_v)  \times \prod_{v\not \in S}  \rho_X^{-1}( g^{-1} \cdot W_v)]\cap Y_0({\bf A}_k) . $$ 
Therefore one obtains $$y\in Y_0(k) \cap [\prod_{v\in S} Y_0(k_v) \times \prod_{v\not\in S} \rho_{X}^{-1}(g^{-1} \cdot W_v)]$$
by Proposition \ref{refine-add}. This implies $g \cdot \rho_{X} (y) \in  W $ as desired. \end{proof}

\section{Appendix}
When $X$ is a sub-variety of an affine space, then $X(k)$ is discrete in $X(\Bbb A_k)$ by the product formula.  Then non-compactness of $\prod_{v\in S}X(k_v)$ is a necessary condition for $X$ satisfying classical strong approximation off $S$. If $\Br(X)/\Br(k)$ is finite, such compactness is still a necessary condition for $X$ satisfying strong approximation with Brauer-Manin obstruction. However, this is no longer true when $\Br(X)/\Br(k)$ is not finite. For example, a torus $T$ is always satisfying strong approximation with Brauer-Manin obstruction off $\infty_k$ by Theorem 2 in \cite{Ha08} whenever $T(k_\infty)$ is compact or not. Semi-simple linear algebraic groups have quite different feature from tori for strong approximation with Brauer-Manin obstruction off $S$ even though $\Br(G)/\Br(k)$ is not finite either when $G$ is not simply connected.

\begin{prop} \label{sass} Let $G$ be a connected semi-simple linear algebraic group over $k$ and $S$ be a non-empty finite subset of $\Omega_k$. Then $G$ satisfies strong approximation with respect to $\Br_1(G)$ (or $\Br(G)$) off $S$ if and only if $\prod_{v\in S}G'(k_v)$ is not compact for any non-trivial simple factor $G'$ of $G^{sc}$.
\end{prop}
 
\begin{proof} 

($\Leftarrow$) Since there is an isogeny 
$$ \pi_G^c: \ G^{sc} \rightarrow G$$ where $G^{sc}$ is a simply connected covering of $G$ over $k$, one has that $\prod_{v\in S}(G^{sc})'(k_v)$ is not compact for any  non-trivial simple factor $(G^{sc})'$ of $G^{sc}$ and the following descent relation 
\begin{equation} \label{d1}  G({\bf A}_k)^{\Br(G)}=G({\bf A}_k)^{\Br_1(G)}=G(k) \cdot \pi_G^c(G^{sc}({\bf A}_k))  \end{equation}
by Theorem 4.3 and Proposition 2.12 and Proposition 2.6 in \cite{CTX} and the functoriality of Braurer-Manin pairing. The result follows from strong approximation for semi-simple simply connected groups off $S$ (see Theorem 7.12 of \S 7.4 Chapter 7 in \cite{PR}).

($\Rightarrow$) Since there is another isogeny $$ \pi_G^d: \ G \rightarrow G^{ad} $$ where $G^{ad}$ is the adjoint group of $G$ over $k$, one obtains 
$$  G^{ad}({\bf A}_k)^{\Br(G^{ad})}=G^{ad}({\bf A}_k)^{\Br_1(G^{ad})}= G^{ad}(k)\cdot [ (\pi_G^{d}\circ \pi_G^{c})(G^{sc}({\bf A}_k))] $$
by (\ref{d1}). Therefore 
$$  G^{ad}({\bf A}_k)^{\Br(G^{ad})} \subseteq G^{ad}(k)\cdot (\pi_G^{d}[ G(k)\cdot \pi_G^{c} (G^{sc}({\bf A}_k))])=G^{ad}(k)\cdot \pi_G^d(G({\bf A}_k)^{\Br(G)}) $$ by (\ref{d1}). On the other hand, it is clear that the right hand side is contained in the left hand side by the functoriality of Brauer-Manin pairing. One obtains the third descent relation
\begin{equation} \label{d3}
G^{ad}({\bf A}_k)^{\Br(G^{ad})} =G^{ad}(k)\cdot \pi_G^d(G({\bf A}_k)^{\Br(G)}) . \end{equation}
 By the assumption that $G$ satisfies strong approximation with respect to $\Br(G)$, one gets that $G^{ad}$ satisfies strong approximation with respect to $\Br(G^{ad})$ by (\ref{d3}).  It is well-known that $G^{ad}$ is a product of simple subgroups of $G^{ad}$ over $k$  (see (1.4.10) Proposition in \cite{Margulis}). By Proposition 3.2 in \cite{LX} (or the same proof replacing $S$ by a non-empty finite subset of $\Omega_k$), one concludes that each simple factor of $G^{ad}$ satisfies strong approximation of the same type. Therefore one can further assume that $G$ is simple.

For any $v\in \Omega_k$, one has the long exact sequence 
$$ 1\rightarrow \ker(\pi_G^c)(k_v)\rightarrow G^{sc}(k_v) \rightarrow G(k_v) \rightarrow H^1(k_v, \ker(\pi_G^c))  $$ by Galois cohomology. 
Since $H^1(k_v, \ker(\pi_G^c))$ is finite by (7.2.6) Theorem in Chapter VII of \cite{NSW} and $\pi_G^c$ is proper, one obtains that $\pi_G^c(G^{sc}(k_v))$ is an open subgroup of $G(k_v)$ for any $v\in \Omega_k$.

Let $T$ be a finite subset of $\Omega_k$ containing $\infty_k$ and $S$ with $T\setminus (\infty_k\cup S)\neq \emptyset$ such that $\pi_G^c$ is extended to an isogeny 
$$\pi_G^c: \ {\bf G}^{sc}\rightarrow {\bf G}$$ of smooth group schemes of finite type over $O_{k,T}$ (see Definition 4 of  \S 7.3 Chapter 7 in \cite{BLR}). For any $v\in T\setminus S$, we choose a non-empty open subset $U_v$ of $G(k_v)$ such that the topological closure $\overline{U}_v$ of $U_v$ is compact and $\overline{U}_v \subset \pi_G^c(G^{sc}(k_v))$.  

If $\prod_{v\in S}G(k_v)$ is compact, then 
$$ \prod_{v\in S}G(k_v)\times \prod_{v\in T\setminus S} \overline{U}_v \times \prod_{v\not\in T} {\bf G}(O_v) $$ is compact in $G({\bf A}_k)$ with $$[\prod_{v\in S} G(k_v)\times \prod_{v\in T\setminus S} U_v \times \prod_{v\not\in T} {\bf G}(O_v) ] \cap G({\bf A}_k)^{\Br(G)}\neq \emptyset $$ by the functoriality of Brauer-Manin pairing.  This implies that $$G(k)\cap [\prod_{v\in S} G(k_v)\times \prod_{v\in T\setminus S} U_v \times \prod_{v\not\in T} {\bf G}(O_v) ] $$ is finite.  Let $x_1, \cdots, x_n$ be all elements in the above finite set. Choose $v_0\in T\setminus S$ and set $$W_{v_0}=U_{v_0}\setminus \{x_1, \cdots, x_n \}. $$ Then the smaller open subset
$$ C= \prod_{v\in S} G(k_v)\times  W_{v_0}\times \prod_{v\in T\setminus (S\cup \{v_0\})} U_v \times \prod_{v\not\in T} {\bf G}(O_v) $$ satisfies that $$C\cap  G(k)= \emptyset \ \ \ \text{but} \ \ \  C\cap G({\bf A}_k)^{\Br_1(G)}\neq \emptyset $$ by the functoriality of Brauer-Manin pairing. This contradicts that $G$ satisfies strong approximation with Brauer-Manin obstruction off $S$. \end{proof}

\bigskip

\noindent{\bf Acknowledgements.}
Special thanks are due to J.-L. Colliot-Th\'el\`ene who suggested significant improvement from the original version. We would like to thank Cyril Demarche, Qing Liu and Philippe Gille for useful suggestion. The first author acknowledges the support of the French Agence Nationale de la Recherche (ANR)
 under reference ANR-12-BL01-0005, and the second author acknowledges the support of  NSFC grant no.11471219.


\end{document}